\documentclass[11pt,a4paper]{amsart}
\usepackage{amssymb}
\usepackage{graphicx}
\usepackage{amsmath}
\usepackage[colorlinks,
linkcolor=blue,
anchorcolor=blue,
citecolor=blue]{hyperref}

\theoremstyle{plain}
\newtheorem{Main}{Theorem}

\newtheorem{Thm}{Theorem}[section]
\newtheorem{Prop}[Thm]{Proposition}
\newtheorem{Lem}[Thm]{Lemma}

\theoremstyle{remark}
\newtheorem{Rem}[Thm]{Remark}
\newtheorem{Def}[Thm]{Definition}

\def\max{\operatorname{max}}
\def\min{\operatorname{min}}

\begin{document}
	
	\title[Upper bound for the box dimension of the hyperbolic dynamics]
	{An upper bound for the box dimension of the hyperbolic dynamics via unstable pressure}
	
	\author{Congcong Qu}
	\address{Congcong Qu, College of Big Data and software Engineering, Zhejiang Wanli University, Ningbo, 315107, Zhejiang, P.R.China}
	\email{congcongqu@foxmail.com}
	\date{\today}

	\thanks{}
	
	\subjclass[2010]{37B10,37B40,37C45,37D20,37D35}
	
	\keywords{box dimension, hyperbolic set, unstable pressure}

	\begin{abstract}
		In this paper, we utilize the sub-additive unstable pressure to give an upper bound for the upper box dimension of the $C^1$ hyperbolic set on unstable manifolds. As a by-product, we give a new expression of the topological pressure. This work is inspired by \cite{FS, WCZ, CliPZ}.
	\end{abstract}
	
	\maketitle
	
	
	\section{Introduction}
	It is still open to find the dimension of the non-conformal repeller and that of the hyperbolic set on unstable manifolds. There are only results of the upper and lower bounds for the dimension of general non-conformal repellers and hyperbolic sets on unstable manifolds. The difficulty is due to the different expanding rate on the different direction for non-conformal ones. For basic knowledge of this topic, one can refer to the book \cite{Pes} and the survey papers \cite{CP,BG} and Chapter 13 in \cite{PV}. The basic tool to study dimension theory is borrowed from thermodynamic formalism, namely the topological pressure. Recently, Hu, Hua and Wu \cite{HHW} generalised the topological entropy to the unstable topological entropy for $C^1$ partially hyperbolic diffeomorphisms to measure the complexity on the unstable direction. Hu, Wu and Zhu \cite{HWZh} further generalized this to the unstable topological pressure and Zhang, Li and Zhou \cite{ZhLZh} to the sub-additive ones. We consider to use sub-additive unstable topological pressure to study the dimension of hyperbolic sets on unstable manifolds.
	
	The motivation of this work is two-fold. The first one is that, recently, Feng and Simon gave an upper bound of the upper box dimension of the non-conformal repeller, which improves the result of \cite{BCH} and \cite{Zh}. They proved this by analysing the geometry of the images of balls under a large number of iterations of $C^1$ maps. By further analysing their technique, we find new expressions of the topological pressure. We apply this to study the dimension of the hyperbolic set on unstable manifold.
	
	The second motivation is that, Wang, Cao and Zhao \cite{WCZ} gave upper bounds of the dimension of the subset of non-conformal repellers via the pressure on the subset, not the pressure on the whole repeller. This inspires us that the pressure on the unstable direction might provide us with enough information to estimate the dimension of the hyperbolic set on unstable manifold. The development of the theory of unstable pressure provides us with a good choice to study this question. The root given by the unstable pressure does not exceed that given by the topological pressure. It is natural to ask whether the root given by the unstable pressure gives an upper bound of the box dimension of the hyperbolic set on unstable manifolds. Besides, Climenhaga, Pesin and Zelerowicz \cite{CliPZ} proved that for the hyperbolic set of a $C^{1+\alpha}$ diffeomorphism and a H$\ddot{\text{o}}$lder continuous potential, the topological pressure on the hyperbolic set and that on the local unstable manifold coincides. We use this result to give an upper bound of the Hausdorff dimension of $C^1$ hyperbolic set on unstable manifold via the structural stability of the hyperbolic set. We state this result in Section \ref{appendix} as an appendix. Thus it might be possible to use the unstable pressure to study the dimension of the hyperbolic set on the unstable manifolds. 
	
	Denote by $\overline{\text{dim}}_B A$ the upper box dimension of the set $A$ and $P^u(f|_{\Lambda},\Phi_f(s))$ the sub-additive topological pressure for sub-additive singular-valued potential $\Phi_f(s)$ (see section \ref{Preliminary} for the details). We state our main result as follows:
	\begin{Main}\label{Main A}
		Let $M$ be a compact Riemannian manifold, $f:M\rightarrow M$ be a $C^1$ diffeomorphism and $\Lambda$ be a locally maximal hyperbolic set for $f$. Then for any $x\in\Lambda$, we have
		$$\overline{dim}_B W_{loc}^u(x)\cap \Lambda \leq s^\ast,$$
		where $s^\ast$ is the unique root of $P^u(f|_\Lambda,\Phi_f(s))$.
	\end{Main}
	
	This paper is organized as follows. In section \ref{Preliminary}, we recall some basic definitions and preliminary results. In section \ref{proof}, we give the details of the proof of the main result. In section \ref{appendix}, we give an upper bound of the Hausdorff dimension of the hyperbolic set on the unstable manifold.
	
	\section{Preliminary}\label{Preliminary}
	
	In this section, we recall some definitions and preliminary results.
	\subsection{Hyperbolic set}
	Let $M$ be a compact Riemannian manifold. Suppose $f$ is a $C^{1}$ diffeomorphism on $M$ and $\Lambda\subset M$ is an $f$-invariant compact set. We say $\Lambda$ is a {\it hyperbolic set } if for any $x\in \Lambda$, the tangent space admits a decomposition $T_{x}M=E^{s}_x\oplus E^{u}_x$ with the following properties:
	\begin{enumerate}
		\item The splitting is $Df$-invariant: $D_{x}f E^{\sigma}_x=E^{\sigma}_{f(x)}$ for $\sigma=s,u$.
		\item The stable subspace $E^{s}_x$ is uniformly contracting and the unstable subspace $E^{u}_x$ is uniformly expanding: there are constants $C\geq 1$ and $0<\lambda<1$ such that for every $n\in \mathbb{N}$ and $v^{s,u}\in E^{s,u}_x$, we have $$\|D_{x}f^{n}(v^{s})\|\leq C\lambda^{n}\|v^{s}\|\quad \text{and} \quad \|D_{x}f^{-n}(v^{u})\|\leq C\lambda^{n}\|v^{u}\|.$$
	\end{enumerate}
	
	Recall that a hyperbolic set $\Lambda$ is {\it locally maximal} if there exists an open neighborhood $U$ of $\Lambda$ such that $\Lambda=\bigcap_{n\in\mathbb{Z}}f^{n}(U)$. $f$ is called {\it topologically transitive} on $\Lambda$, if for any two nonempty (relative) open subsets $U,V\subset\Lambda$ there exists $n>0$ such that $f^n(U)\cap V\neq\varnothing$.
	Assume that $\Lambda$ is a locally maximal hyperbolic set and $f$ is topologically transitive on $\Lambda$. 
	A finite cover $\mathcal{P}=\{P_{1},P_{2},...,P_{l}\}$ of $\Lambda$  is called a {\it Markov partition} of $\Lambda$ with respect to $f$ if
	\begin{enumerate}
		\item $\varnothing \neq P_{i} \subset \Lambda$ and $P_{i}=\overline{\text{int } P_{i}}$ for each $i=1,2,...,l$;
		\item ($\text{int } P_{i}) \cap (\text{int } P_{j} )=\varnothing$ whenever $i\neq j$;
		\item for every $x, y \in P_i$, 
		\[ 
		[x, y]: =W_{\text{loc}}^s(x,f) \cap W_{\text{loc}}^u(y,f)
		\] 
		exists and is contained in $P_i$	, $i=1,2,\cdots,l$;
		\item for each $x\in (\text{int } P_{i}) \cap f^{-1} (\text{int } P_{j})$, we have
		$$f(W^{s}_{\text{loc}}(x, f)\cap P_{i})\subset W^{s}_{\text{loc}}(f(x), f)\cap P_{j},$$
		$$f(W^{u}_{\text{loc}}(x, f)\cap P_{i})\supset W^{u}_{\text{loc}}(f(x), f)\cap P_{j}.$$
	\end{enumerate}
	Here the topology used for the interior is the induced topology on $\Lambda$.
	

	Let $\mathcal{P}=\{P_{1},P_{2},...,P_{l}\}$ be a Markov partition of $\Lambda$ with diameter as small as desired (we refer to \cite{KH} for details and references). Denote $\mathcal{A}=\{1,2,...,l\}$. We equip the space of sequences $\Sigma_{l}=\mathcal{A}^{\mathbb{Z}}$ with the distance
	\begin{align*}
		d(w,w^{\prime})\triangleq \sum_{j\in\mathbb{Z}}e^{-|j|}|i_{j}-i^{\prime}_{j}|,
	\end{align*}
	where $w=(\cdots i_{-1}i_{0}i_{1}\cdots)$ and $w^{\prime}=(\cdots i^{\prime}_{-1}i^{\prime}_{0}i^{\prime}_{1}\cdots)$. With this distance, $\Sigma_l$ becomes a compact metric space. We also consider the shift map $\sigma: \Sigma_{l}\rightarrow \Sigma_{l}$ defined by $(\sigma(w))_{j}=i_{j+1}$ for each $w=(\cdots i_{-1}i_{0}i_{1}\cdots)\in \Sigma_{l}$ and $j\in\mathbb{Z}$. The restriction of $\sigma$ to the set
	\begin{align*}
		\Sigma_{A}\triangleq\{(\cdots i_{-1}i_{0}i_{1}\cdots)\in \Sigma_{l}:a_{i_{j}i_{j+1}}=1 \text{~for all~} j\in\mathbb{Z}\}
	\end{align*}
	is the two-sided {\it topological Markov chain} with the transition matrix $A=(a_{ij})$ of the Markov partition (i.e. $a_{ij}=1$ if $(\text{int } P_{i})\cap f^{-1}(\text{int } P_{j})\neq \varnothing$ and $a_{ij}=0$ otherwise). It is easy to show that one can define a coding map $\pi:\Sigma_{A}\rightarrow \Lambda$ of the hyperbolic set $\Lambda$ by
	\begin{align*}
		\pi(\cdots i_{-n}\cdots i_{0}\cdots i_{n}\cdots)=\bigcap_{n\in \mathbb{Z}}f^{-n}(P_{i_{n}}).
	\end{align*}
	The map $\pi$ is surjective and satisfies $\pi\circ\sigma=f\circ\pi$.
	
	Let $X\subset \Sigma_l$ be a non-empty compact set satisfying $\sigma X=X$. We call $(X,\sigma)$ a {\it two-sided subshift over $\mathcal{A}$}.

	A sequence $\mathbf{i}=(\cdots i_{-m}\cdots i_{0}\cdots i_{n})$, where $1\leq i_{j}\leq l$, is called {\it admissible} if for $j=...,-m,...,0,...,n-1$, it holds that $(\text{int } P_{i_j}) \cap f^{-1}(\text{int } P_{i_{j+1}})\neq \varnothing$. Given an admissible sequence $\mathbf{i}=(\cdots i_{-m}\cdots i_{0}\cdots i_{n})$, one can define the cylinder
	$$P_{\cdots i_{-m}\cdots i_{0}\cdots i_{n}}=\bigcap_{j=-\infty}^{n}f^{-j}(P_{i_{j}}).$$
	\subsection{Partially hyperbolic diffeomorphism} Let $M$ be a compact Riemannian manifold and $f:M\rightarrow M$ be a $C^1$ diffeomorphism. $f$ is said to be a {\it partially hyperbolic diffeomorphism} if there exists a nontrivial $Df-$invariant splitting $T_xM=E^s_x\oplus E^c_x \oplus E^u_x$ of the tangent bundle into stable, center and unstable distributions such that for all $x\in M$ and $v^\sigma$($\sigma=c,s,u$), it holds that 
	$$\|D_xf(v^s)\|<\|D_xf(v^c)\|<\|D_xf(v^u)\|$$
	and
	$$\|D_xf|_{E_x^s}\|<1\\\text{\quad and\quad } \|D_xf^{-1}|_{E^u_x}\|<1$$
	for some suitable Riemannian metric on $M$. The stable distribution $E^s_x$ and unstable distribution $E^u_x$ are integrable to the stable and unstable foliations $W^s(x)$ and $W^u(x)$ respectively such that $TW^s(x)=E^s_x$ and $TW^u(x)=E^u_x$.
	\subsection{Box dimension}
	Let $X$ be a compact metric space. {\em The lower and upper box dimensions of the set $Z\subset X$} are defined respectively by
	\[
	\underline{\dim}_B Z=\liminf_{\varepsilon\to 0}\frac{\log N(Z, \varepsilon)}{\log \frac{1}{\varepsilon}} \text{ \quad and \quad} \overline{\dim}_B Z=\limsup_{\varepsilon\to 0}\frac{\log N(Z, \varepsilon)}{\log \frac{1}{\varepsilon}},
	\]
	where $N(Z, \varepsilon)$ denotes the least number of balls of radius $\varepsilon$ those are needed to
	cover the set $Z$. For more introduction, one can refer to \cite{Pes}.

	\subsection{Topological pressure}\label{Topological pressure}
	We recall the definition of the topological pressure. For a detailed introduction, one can refer to \cite{Pes,WALTERS,CFH}.
	Let $X$ be a compact metric space equipped with the metric $d$ and $f:X\rightarrow X$ be a continuous map. Given $n\in\mathbb{N}$, define the metric $d_{n}$ on $X$ as
	\begin{align*}
		d_{n}(x,y)=\max\{d(f^{k}(x),f^{k}(y)):0\leq k<n\}.
	\end{align*}
	Given $r>0$, denote by $B_{n}(x,r)=\{y:d_{n}(x,y)<r\}$ the {\it Bowen balls}. We say that a set $E\subset X$ is $(n,r)${\it -separated} for $X$ if $d_n(x,y)>r$ for any two distinct points $x,y\in E$.
	A sequence of continuous potential functions $\Phi=\{\varphi_n\}_{n\geq1}$ is called {\it sub-additive}, if
	\begin{equation*}
		\varphi_{m+n}\leq\varphi_n+\varphi_m\circ f^n, \text{ for all } m,n\geq1.
	\end{equation*}

	Let $\Phi=\{\varphi_n\}_{n\geq1}$ be a sub-additive sequence of continuous potentials on $X$, set
	\begin{equation*}
		P_n(\Phi, \varepsilon)=\sup\{\sum_{x\in E}e^{\varphi_n(x)}: E \text{ is an } (n,\varepsilon) \text{-separated subset of } X\}.
	\end{equation*}
	The quantity
	\begin{equation*}
		P(f,\Phi)=\lim_{\varepsilon\to0}\limsup_{n\to\infty}\frac1n\log P_n(\Phi,\varepsilon)
	\end{equation*}
	is called the {\it sub-additive topological pressure} of $\Phi$.
	
	The sub-additive topological pressure satisfies the following  variational principle, see \cite{CFH}.
	\begin{Thm}\label{variational principle}
		Let $\Phi=\{\varphi_n\}_{n\geq1}$ be a sub-additive sequence of continuous potentials on a compact dynamical system $(X,f)$. Then
		\begin{equation*}
			P(f,\Phi)=\sup\{h_{\mu}(f)+\Phi_*(\mu)| ~\mu\in\mathcal{M}_{inv}(f), \Phi_*(\mu)\neq-\infty\},
		\end{equation*}
		where $\Phi_*(\mu)=\lim_{n\to\infty}\frac1n\int\varphi_n d\mu$ (the equality is due to the standard sub-additive argument) and $\mathcal{M}_{inv}(f)$ is the set of all $f$-invariant Borel probability measures on X.
	\end{Thm}
	
	\begin{Rem}
		If $\Phi=\{\varphi_n\}_{n\geq1}$ is \emph{additive} in the sense that $\varphi_n(x)=\varphi(x)+\varphi(f(x))+\cdots+\varphi(f^{n-1}(x))\triangleq S_n\varphi(x)$ for some continuous function $\varphi: X\to \mathbb{R}$, we simply denote the topological pressure $P(f, \Phi)$ as $P(f, \varphi)$.
	\end{Rem}

	Next we recall the definition of the singular-valued potentials for hyperbolic diffeomorphisms.
	
	Suppose $M$ is a compact Riemannian manifold and $f:M\rightarrow M$ is a $C^{1}$ diffeomorphism admitting a hyperbolic set $\Lambda$ such that $f|_{\Lambda}$ is topologically transitive. Given $x\in \Lambda$ and $n\geq 1$, we consider the differentiable operator $D_{x}f^{n}|_{E^u_x}: E^u_x\rightarrow E^u_{f^{n}(x)}$ and denote its singular values in the decreasing order by
	$$\alpha_{1}(x,f^{n})\geq\alpha_{2}(x,f^{n})\geq\cdots\geq \alpha_{u}(x,f^{n}),$$ where $u$ is the dimension of the unstable manifold.
	For $s\in [0,u]$, set
	$$\varphi^{s}(x,f^{n})\triangleq\sum^{u}_{i=u-[s]+1} \log \alpha_{i}(x,f^{n})+(s-[s])\log\alpha_{u-[s]}(x,f^{n}).$$
	
	One can show that $\Phi_{f}(s)\triangleq \{-\varphi^{s}(\cdot,f^{n})\}_{n\geq 1}$ are sub-additive. We call them the {\it sub-additive singular-valued potentials}.
	\subsection{Unstable topological pressure}
	The notion of unstable topological pressure is given by Hu, Wu and Zhu in \cite{HWZh} and generalized to sub-additive ones by Zhang, Li and Zhou in \cite{ZhLZh}. Given a partially hyperbolic diffeomorphism $f$ and a sequence of sub-additive potentials $\mathcal{\varPhi}=\{\varphi_n\}_{n\geq 1}$. Denote by $d^u$ the metric induced by the Riemannian structure on the unstable manifold and let $d^u_n(x,y)=\max_{0\leq j\leq n-1}d^u(f^j(x),f^j(y))$. Let $W^u(x,\delta)$ denote the open ball inside $W^u(x)$ with center $x$ and radius $\delta$ with respect to $d^u$. Let $E$ be the set of points in $\overline{W^u(x,\delta)}$ with pairwise $d^u_n-$distances at least $\varepsilon$. The set $E$ is called $(n,\varepsilon)$ $u$-{\it separated subset} of $\overline{W^u(x,\delta)}$. Set
	\begin{equation*}
		P^u_n(f,\mathcal{\varPhi},\varepsilon,x,\delta)=\sup\{\sum_{y\in E}e^{\varphi_n(y)}|E\ \text{is}\ \text{an} \ (n,\varepsilon)\ u-\text{separated subset of }\overline{W^u(x,\delta)}\}
	\end{equation*}
	and
	\begin{equation*}
		P^u(f,\varPhi,\varepsilon,x,\delta)=\limsup_{n\rightarrow \infty}\frac{1}{n}\log P^u_n(f,\varPhi,\varepsilon,x,\delta).
	\end{equation*}
	Set
	\begin{equation*}
		P^u(f,\varPhi,x,\delta)=\lim_{\varepsilon\rightarrow 0}P^u(f,\varPhi,\varepsilon,x,\delta).
	\end{equation*}
	The \emph{unstable topological pressure} of $f$ with respect to $\varPhi$ is defined as 
	\begin{equation*}
		P^u(f,\varPhi)=\lim_{\delta\rightarrow 0}\sup_{x\in M}P^u(f,\varPhi,x,\delta).
	\end{equation*} 
	If $\varphi_n$ is additive, that is $\varphi_n(x)=\varphi(x)+\varphi(f(x))+\cdots+\varphi(f^{n-1}(x))$ for some continuous potential $\varphi$, then we denote $	P^u(f,\varPhi)$ as $P^u(f,\varphi)$.
	\section{Proof of Theorem A}\label{proof}
	\subsection{Topological pressure and unstable pressure}
	First, we investigate the relation of sub-additive topological pressure and unstable pressure. 
	The main result in this subsection is the following theorem. In this subsection, we assume that $M$ is a compact Riemannian manifold.
	\begin{Thm}\label{equal}
		Suppose $f:M \rightarrow M$ is a $C^{1}$ diffeomorphism, $\Lambda$ is a hyperbolic set for $f$ and $\Phi_f(s)$ are the sub-additive singular-valued potentials. Then it holds that
		$$P^{u}(f|_{\Lambda},\Phi_f(s))=P(f|_{\Lambda},\Phi_f(s)).$$
	\end{Thm}
	The result in \cite{HWZh} shows that
	\begin{Prop}
		Let $f:M\rightarrow M$ be a $C^1$ partially hyperbolic diffeomorphism and $\varphi:M\rightarrow \mathbb{R}$ be a continuous potential. Then
		$$P^u(f,\varphi)\leq P(f,\varphi).$$
		If $f$ is $C^{1+\alpha}$ and there is no positive Lyapunov exponent in the center direction at $\nu-$a.e. with respect to any ergodic measure $\nu$, then the equation holds.
	\end{Prop}
	By this proposition, for the hyperbolic set $\Lambda$ for a $C^{1+\alpha}$ hyperbolic diffeomorphism $f$ and a continuous potential $\varphi$, it holds that $$P^{u}(f|_{\Lambda},\varphi)=P(f|_{\Lambda},\varphi).$$
	
	It is natural to ask whether this holds for hyperbolic sets for $C^1$  diffeomorphisms and continuous potentials. In our work, we show that it holds for singular-valued potential.

	\begin{Thm}\label{unstablepressure}
		Suppose $g:M\rightarrow M$ is a $C^{1}$ diffeomorphism, $\Lambda_{g}$ is a hyperbolic set for $g$. Denote by $u$ the dimension of the unstable manifold. Then for every $n\in\mathbb{N}$ and $s\in[0,u]$, we have 
		$$P(g|_{\Lambda_{g}},-\frac{1}{n}\varphi^{s}(\cdot,g^{n}))=P^{u}(g|_{\Lambda_{g}},-\frac{1}{n}\varphi^{s}(\cdot,g^{n})).$$
	\end{Thm}
	
	We derive this via the structural stability of $C^{1}$ hyperbolic set. One can refer to Theorem $18.2.1$ in \cite{KH} for the structural stability. We state it as following:
	\begin{Thm}\label{structure stability}
		Given a compact Riemannian manifold $M$ and an open set $U\subset M$. Suppose $f:U\rightarrow f(U)$ is a $C^{1}$ diffeomorphism and $\Lambda\subset U$ is a hyperbolic set for $f$. Then for any open neighborhood $V\subset U$ of $\Lambda$ and $\varepsilon>0$, there exists $\delta>0$, such that for any diffeomorphism $g:U\rightarrow M$ with $d_{C^{1}}(f|_{V},g)<\delta$, there exists a hyperbolic set $\Lambda_{g}\subset V$ for $g$ and a homeomorphism $h:\Lambda_{g}\rightarrow \Lambda$ satisfying
		\begin{enumerate}	
			\item $d_{C^{0}}(Id,h)+d_{C^{0}}(Id,h^{-1})<\varepsilon$
			\item $h\circ g|_{\Lambda_{g}}=f|_{\Lambda}\circ h$ 
		\end{enumerate}
		and  $h$ is unique when $\varepsilon$ is small enough.
	\end{Thm}
	\begin{proof}[Proof of Theorem \ref{unstablepressure}]
		For $y\in M$ and $n\in\mathbb{N}$, denote the singular values of $D_yg^{n}$ decreasingly by
		$$\alpha_{1}(y,g^{n})\geq\cdots\geq\alpha_{d}(y,g^{n}),$$
		where $d$ is the dimension of the manifold $M$.
		For every $s\in [0,u]$, denote
		$$\varphi^{s}(y, g^{n})\triangleq\sum_{i=u-[s]+1}^{u}\log \alpha_{i}(y,g^{n})+(s-[s])\log\alpha_{u-[s]}(y,g^{n}).$$
		For $\varepsilon_{1}>0$, there exists $\delta>0$, such that for any $x_{1},x_{2}\in\Lambda_{g}$ satisfying  $d(x_{1},x_{2})\leq\delta$, we have 	$$|\frac{1}{n}\varphi^{s}(x_{1},g^{n})-\frac{1}{n}\varphi^{s}(x_{2},g^{n})|\leq \frac{\varepsilon_{1}}{4}.$$
		By Theorem \ref{structure stability}, for $\delta>0$ above, there exists $\varepsilon>0$ such that for any diffeomorphism $f:M\rightarrow M$ with $d_{C^{1}}(f,g)<\varepsilon$ and $d_{C^{1}}(f^n,g^n)<\varepsilon$, there exists  a hyperbolic set $\Lambda_{f}$ for $f$ and a homeomorphism $h:\Lambda_{f}\rightarrow \Lambda_{g}$, such that $d_{C^{0}}(Id,h)+d_{C^{0}}(Id,h^{-1})<\delta$ and for $x\in \Lambda_f$, one has $$|\frac{1}{n}\varphi^{s}(x,g^{n})-\frac{1}{n}\varphi^{s}(x,f^{n})|\leq \frac{\varepsilon_{1}}{4}.$$
		Thus
		\begin{align*}
			&\max_{x\in\Lambda_{f}}|\frac{1}{n}\varphi^{s}(h(x),g^{n})-\frac{1}{n}\varphi^{s}(x,f^{n})|\\
			\leq & \max_{x\in\Lambda_{f}}(|\frac{1}{n}\varphi^{s}(h(x),g^{n})-\frac{1}{n}\varphi^{s}(x,g^{n})|+|\frac{1}{n}\varphi^{s}(x,g^{n})-\frac{1}{n}\varphi^{s}(x,f^{n})|)\\
			\leq &\frac{\varepsilon_{1}}{2}.
		\end{align*}
		Choose a $C^{1+\alpha}$ diffeomorphism $f$ satisfying~$d_{C^{1}}(f,g)<\varepsilon$ and $d_{C^{1}}(f^n,g^n)<\varepsilon$. Then we have
		\begin{equation}\label{pressure relation1}
			P(f|_{\Lambda_{f}},-\frac{1}{n}\varphi^{s}(\cdot,f^{n}))=P^u(f|_{\Lambda_f},-\frac{1}{n}\varphi^{s}(\cdot,f^{n})).
		\end{equation}	
		By the property of topological pressure and unstable pressure, we have 
		\begin{equation}\label{pressure relation2}
			P(g|_{\Lambda_{g}},-\frac{1}{n}\varphi^{s}(\cdot,g^{n}))=P(f|_{\Lambda_{f}},-\frac{1}{n}\varphi^{s}(h(\cdot),g^{n}))
		\end{equation}
		and 
		\begin{equation}\label{pressure relation3}
			P^u(g|_{\Lambda_{g}},-\frac{1}{n}\varphi^{s}(\cdot,g^{n}))
			=P^u(f|_{\Lambda_{f}},-\frac{1}{n}\varphi^{s}(h(\cdot),g^{n})).
		\end{equation}
		By (\ref{pressure relation1}), (\ref{pressure relation2}) and (\ref{pressure relation3}), we have
		\begin{align*}
			&|P(g|_{\Lambda_{g}},-\frac{1}{n}\varphi^{s}(\cdot,g^{n}))-P^u(g|_{\Lambda_{g}},-\frac{1}{n}\varphi^{s}(\cdot,g^{n}))|\\
			=&|P(f|_{\Lambda_{f}},-\frac{1}{n}\varphi^{s}(h(\cdot),g^{n}))-P^u(g|_{\Lambda_{g}},-\frac{1}{n}\varphi^{s}(\cdot,g^{n}))|\\
			\leq &|P(f|_{\Lambda_{f}},-\frac{1}{n}\varphi^{s}(h(\cdot),g^{n}))-P(f|_{\Lambda_{f}},-\frac{1}{n}\varphi^{s}(\cdot,f^{n}))|\\
			+&|P(f|_{\Lambda_{f}},-\frac{1}{n}\varphi^{s}(\cdot,f^{n}))-P^u(f|_{\Lambda_{f}},-\frac{1}{n}\varphi^{s}(\cdot,f^{n}))|\\
			+&|P^u(f|_{\Lambda_{f}},-\frac{1}{n}\varphi^{s}(\cdot,f^{n}))-P^u(f|_{\Lambda_{f}},-\frac{1}{n}\varphi^{s}(h(\cdot),g^{n}))|\\
			+&|P^u(f|_{\Lambda_{f}},-\frac{1}{n}\varphi^{s}(h(\cdot),g^{n}))-P^u(g|_{\Lambda_{g}},-\frac{1}{n}\varphi^{s}(\cdot,g^{n}))|\\
			\leq &2\max_{x\in\Lambda_{f}}|\frac{1}{n}\varphi^{s}(h(x),g^{n})-\frac{1}{n}\varphi^{s}(x,f^{n})|\leq \varepsilon_{1}.
		\end{align*}
		Letting $\varepsilon_{1}$ tend to $0$, we finish the proof.
	\end{proof}
	The result in \cite{BCH} shows that the sub-additive topological pressure can be approximated by a sequence of additive ones.
	\begin{Prop}\label{topo pressure appro}
		Suppose $f:M\rightarrow M$ is a $C^1$ local diffeomorphism  with  the map $h_{\mu}(f):\mathcal{M}_{inv}(f)\rightarrow \mathbb{R}$ is upper semi-continuous and $\Phi=\{\varphi_n(x)\}$ is a sub-additive continuous function sequence on $M$. Then we have 
		$$P(f,\Phi)=\lim_{n\rightarrow \infty} P(f,\frac{\varphi_n}{n})=\lim_{n\rightarrow \infty}\frac{1}{n}P(f^n,\varphi_n).$$
	\end{Prop}
	A similar result holds for the sub-additive unstable pressure.
	\begin{Prop}\cite{ZhLZh}\label{unstable pressure appro}
		Let $f:M\rightarrow M$ be a $C^1$ partially hyperbolic diffeomorphism and $\Phi=\{\varphi_n(x)\}$ be sub-additive potentials of $f$ on $M$. Then 
		$$P^u(f,\Phi)=\lim_{n\rightarrow \infty} P^u(f,\frac{\varphi_n}{n})$$
		if and only if 
		\begin{align}\label{variational principle2}
			P^u(f,\Phi)=\sup \{h_\mu^u(f)+\Phi_\ast(\mu):\mu\in \mathcal{M}_{inv}(f)\},
		\end{align}
		where $\mathcal{M}_{inv}(f)$ denotes the space of the invariant measures for $f$ and $$\Phi_{\ast}(\mu)=\lim_{n\rightarrow\infty}\frac{1}{n}\int \varphi_n d\mu.$$
	\end{Prop}
	\begin{proof}[Proof of Theorem \ref{equal}]
		By Theorem \ref{unstablepressure}, it holds that $$P(f|_{\Lambda_{f}},-\frac{1}{n}\varphi^{s}(\cdot,f^{n}))=P^{u}(f|_{\Lambda_{f}},-\frac{1}{n}\varphi^{s}(\cdot,f^{n})).$$By Theorem 1.1 in \cite{ZhLZh}, (\ref{variational principle2}) holds for $C^1$ partially hyperbolic diffeomorphisms. Thus by Proposition \ref{topo pressure appro} and \ref{unstable pressure appro}, it holds that 
		\begin{align*}
			P^{u}(f|_{\Lambda_f},\Phi_f(s))&=\lim_{n \rightarrow \infty}P^{u}(f|_{\Lambda_{f}},-\frac{1}{n}\varphi^{s}(\cdot,f^{n}))\\
			&=\lim_{n \rightarrow \infty}P(f|_{\Lambda_{f}},-\frac{1}{n}\varphi^{s}(\cdot,f^{n}))\\
			&=P(f|_{\Lambda_f},\Phi_f(s)).
		\end{align*}
	\end{proof}
	
	\subsection{A new definition for topological pressure} While studying the dimension of the hyperbolic set on the unstable manifold, we find a new expression of the topological pressure for a special potential and conjecture that this holds for general continuous potentials, see Proposition \ref{unique root}. Besides, we give a new definition of the topological pressure via the dimensional form, see Proposition \ref{new formula}. This is inspired by Proposition 3.4 in \cite{FS}. We recall the definition of topological pressure on subsets, one can see \cite{Pes} for reference. 
	\begin{Def}
		Suppose $X$ is a compact metric space and $f: X \rightarrow X$ is a continuous map.
		For any given potential  $\varphi: X\rightarrow \mathbb{R}$, subset  $Z\subset X$, $\delta >0$ and  $N\in \mathbb{N}$, denote  $\mathcal{P}(Z, N, \delta)$ the collection of countably many sets  $\Big\{(x_{i}, n_{i})\in Z\times \{N, N+1, ...\}\Big\}$ such that
		$Z\subset \bigcup_{i}B_{n_{i}}(x_{i}, \delta)$,
		where
		$$ B_{n_{i}}(x_{i}, \delta)=\{y\in X: d(f^{j}(x_{i}), f^{j}(y))<\delta, \ j=0, 1, ..., n_{i}-1\}.$$
		Given $ s\in \mathbb{R}$, define
		\begin{eqnarray*}
			\begin{aligned}
				m_{P}(Z, s, \varphi, N, \delta)&=\inf_{\mathcal{P}(Z, N, \delta)}\sum_{(x_{i}, n_{i})}\exp(-n_{i}s+S_{n_{i}}\varphi(x_{i})),\\
				m_{P}(Z, s, \varphi, \delta)&=\lim_{N\rightarrow \infty}m_{P}(Z, s, \varphi, N, \delta).
			\end{aligned}
		\end{eqnarray*}
		$m_{P}(Z, s, \varphi, \delta)$ is non-increasing in ~$s$, and takes values $\infty$ and $0$ at all but at most one value of $s$. Denote the critical value of $s$ by
		\begin{align*}
			P_{Z}(\varphi, \delta)&=\inf\{s\in \mathbb{R}|m_{P}(Z, s, \varphi, \delta)=0\}\\
			&=\sup\{s\in \mathbb{R}|m_{P}(Z, s, \varphi, \delta)=\infty\}.
		\end{align*}
		
		The {\it topological pressure of the potential ~$\varphi$ with respect to $f$ on a set $Z$} is defined as
		$$P_{Z}(f, \varphi)=\lim_{\delta \rightarrow 0}P_{Z}(\varphi, \delta).$$
	\end{Def}
	If $Z$ is a compact invariant set, then this definition coincides with that in subsection \ref{Topological pressure}, see \cite{Pes} for reference.
	
	To give the new definition of the topological pressure, we give some notations first, which are similar to the ones in the definition of the topological pressure on subsets. Suppose $(X,\sigma)$ is a two-sided subshift over a finite alphabet $\mathcal{A}$ and $\varphi$ and $h$ are continuous functions. Assume that $h(x)<0$ for all $x\in X$. 
	For $r>0$, we set 
	$$\mathcal{A}_r\triangleq \{i_0...i_{n-1}\in X^\ast:\sup_{x\in[i_0...i_{n-1}]\cap X}\exp{(S_nh(x))}<r\leq \sup_{y\in[i_0...i_{n-2}]\cap X}\exp{(S_{n-1}h(y))}\},$$
	where $X^*$ is the finite words allowed in $X$. For a finite word $I=i_0...i_{n-1}\in X^*$, denote $|I|=n$ the length of $I$ and $[I]$ the corresponding cylinder.  
	For $r>0$ and $\lambda\in\mathbb{R}$, we set
	$$\tilde{m}_P(Z,\lambda,\varphi,r)\triangleq \inf\sum_{|I_i|}\exp{(-\lambda |I_i|)}\sup_{x\in [I_i]\cap X}\exp{S_{|I_i|}\varphi(x)} $$
	where the infimum is taken over all the cover of $Z$ of the form $\{[I_i]\}_{I_i\in \mathcal{A}_{r_i}}$ with $r_i\leq r$. Denote by
	$$\tilde{m}_P(Z,\lambda,\varphi)\triangleq \lim_{r \rightarrow 0} \tilde{m}_P(Z,\lambda,\varphi,r)$$ and
	$$\tilde{P}_Z(\sigma,\varphi)\triangleq \inf\{\lambda\in\mathbb{R}|\tilde{m}_P(Z,\lambda,\varphi)=0\}.$$

	The first main result in this subsection is the following proposition.
	\begin{Prop}\label{new formula}
		Assume that $(X,\sigma)$ is a two-sided subshift over a finite alphabet $\mathcal{A}$. Suppose $\varphi$ and $h$ are continuous functions. Assume that $h(x)<0$ for all $x\in X$. Then we have
		$$\tilde{P}_X(\sigma,\varphi)=P_X(\sigma,\varphi).$$
	\end{Prop}
	\begin{proof}
		For symbolic dynamical systems, the definition of the topological pressure on subsets can be simplified as the following, see for example Appendix II in \cite{Pes}. 
		Given $Z\subset X$, $n\in \mathbb{N}$ and $\lambda\in\mathbb{R}$, set
		$$m_P(Z,\lambda,\varphi,n)\triangleq \inf\sum_{|I_i|\triangleq n_i \geq n}\exp{(-\lambda |I_i|)}\sup_{x\in [I_i]\cap X}\exp{S_{|I_i|}\varphi(x)}, $$
		where the infimum is taken over all the cover $\{[I_i]\}_{|I_i|\triangleq n_i \geq n}$ of $Z$. Set
		$$m_P(Z,\lambda,\varphi)\triangleq \lim_{n \rightarrow +\infty} m_P(Z,\lambda,\varphi,n)$$ and
		$$P_Z(\sigma,\varphi)\triangleq \inf\{\lambda\in\mathbb{R}|m_P(Z,\lambda,\varphi)=0\}.$$

		For any $r>0$ small and $\bold{i}\in X$, there exists $n(\bold{i},r)\in \mathbb{N}$ such that $i_0i_1...i_{{n(\bold{i},r)}-1}\in\mathcal{A}_r$.  Set
		\begin{equation}\label{least1}
			m(r)\triangleq\inf\{|I|:I\in\mathcal{A}_r\}.
		\end{equation}
		For any $r_0>0$, there exists $N(r_0)\triangleq m(r_0)$ such that for any $0<r<r_0$, any $\bold{i}\in X$, we have 
		\begin{equation*}
			n(\bold{i},r)\geq n(\bold{i},r_0)\geq N(r_0). 
		\end{equation*}
		For $r>0$ small enough and a cover $\{[I_i]\}_{I_i\in \mathcal{A}_{r_i}}$ of $X$ with $r_i\leq r$, set $N=m(r)$, then $|I_i|\triangleq n_i\geq N$.
		Thus  $\{[I_i]\}_{I_i\in \mathcal{A}_{r_i}}$ is also  a cover of $X$ with $|I_i|\triangleq n_i\geq N$ and $N$ tends to infinity as $r$ tends to $0$. It follows that  $$m_P(X,\lambda,\varphi)\leq \tilde{m}_P(X,\lambda,\varphi).$$ 
		Thus we have $P_X(\sigma,\varphi)\leq \tilde{P}_X(\sigma,\varphi)$.

		Let $r_1=\sup_{x\in X}\exp(h(x))$. For $r\in (0,r_1)$, define $$\Gamma_r=\sum_{I\in\mathcal{A}_r}\sup_{x\in[I]\cap X}\exp(S_{|I|}\varphi(x)).$$
		Denote $P=P_X(\sigma,\varphi)$. By definition we have 
		$$\lim_{n\rightarrow +\infty}\frac{1}{n}\log(\sum_{I\in X^*:|I|=n}\sup_{x\in[I]\cap X}\exp{(S_{|I|}\varphi(x))})=P.$$
		Hence for any $\lambda>0$, there exists a large $N$ such that $e^{-\lambda N/2}<1-e^{-\lambda/2}$ and for any $n\geq N$,
		$$\gamma_n\triangleq \sum_{I\in X^*:|I|=n}\exp((-\lambda-P)|I|)\sup_{x\in [I]\cap X}\exp{(S_{|I|}\varphi(x))}\leq \exp{(-\lambda n/2)}.$$
		Since $m(r)$ tends to infinity as $r$ tends to $0$, where $m(r)$ is defined in (\ref{least1}), we can take a small $r_2\in (0,r_1)$ so that $m(r)\geq N$ for any $0<r<r_2$. Hence for any $0<r<r_2$, $$\mathcal{A}_r\subset \{I\in X^*:|I|\geq N\}$$ and so 
		\begin{align*}
			&\sum_{I\in\mathcal{A}_r}\exp{((-\lambda-P)|I|)}\sup_{x\in[I]\cap X}\exp(S_{|I|}\varphi(x)) \\
			\leq &\sum_{n=N}^{\infty}\gamma_n \leq \sum_{n=N}^{\infty}\exp{(-\lambda n/2)}=\frac{e^{-\lambda N/2}}{1-e^{-\lambda/2}}<1.
		\end{align*}
		
		It follows that $\tilde{m}_P(X,\lambda+P,\varphi)\leq 1$. By arbitrariness of $\lambda>0$ and the definition of $\tilde{P}_X(\sigma,\varphi)$, we have $\tilde{P}_X(\sigma,\varphi)\leq P$.
	\end{proof}

	We also can give another expression of topological pressure for a special potential. This result is inspired by an observation of the proof of Proposition 3.4 in \cite{FS}.
	\begin{Prop}\label{unique root}
		Assume that $(X,\sigma)$ is a two-sided subshift over a finite alphabet $\mathcal{A}$. Suppose $g$ and $h$ are continuous functions. Assume that $h(x)<0$ for all $x\in X$ and $t\in\mathbb{R}$ is the unique root of $P_X(\sigma,g+th)=0$. Then
		\begin{eqnarray*}
			\lim_{r\rightarrow 0}\frac{\log \Gamma_r}{\log r}=0,
		\end{eqnarray*}
		where $$\Gamma_r=\sum_{I\in \mathcal{A}_r}\sup_{x\in [I]\cap X} \exp{(S_{|I|}(g+th)(x))}.$$
	\end{Prop}

	\begin{proof}
		Let $r_0=\sup_{x\in X}\exp(h(x))$. Set for $0<r<r_0$, 
		\begin{equation}\label{minimal}
			m(r)=\min\{|I|:I\in \mathcal{A}_r\},\quad M(r)=\max\{|I|:I\in\mathcal{A}_r\}.
		\end{equation}
		From the definition of $\mathcal{A}_r$ and the negativity of $h$, it follows that there exist two positive constants $a,b$ such that 
		\begin{equation}\label{range}
			a\log(\frac{1}{r})\leq m(r)\leq M(r)\leq b\log (\frac{1}{r}),\forall r\in (0,r_0).
		\end{equation}
		In detail, for each $0<r<r_0$ and $i_0...i_{n-1}\in\mathcal{A}_r$, by definition of $\mathcal{A}_r$,
		\begin{align}\label{size}
			\inf_{y\in[i_0...i_{n-2}]\cap X}\exp{S_{n-1}(-h)(y)}\leq \frac{1}{r}<\inf_{x\in[i_0...i_{n-1}]\cap X}\exp{S_{n}(-h)(x)}.
		\end{align}
		By the negativity of $h$, there exists $a_1,a_2>0$ such that $$a_1<-h(x)<a_2$$ for all $x\in X$. Thus $$\frac{1}{r}<\inf_{x\in[i_0...i_{n-1}]\cap X}\exp{S_{n}(-h)(x)}<\exp{na_2}$$ for all $i_0...i_{n-1}\in\mathcal{A}_r$. It follows that $$(\frac{1}{r})^{\frac{1}{a_2}}<\exp{m(r)}.$$
		By (\ref{size}), for each $i_0...i_{n-1}\in\mathcal{A}_r$, $$\frac{1}{r}\geq \exp{S_{n-1}(-h)(y)}+\frac{\exp{(n-1)a_1}-\exp{S_{n-1}(-h)(y)}}{2}>\exp{(n-1)a_1}$$ for some $y\in[i_0...i_{n-2}]\cap X$. Thus $$(\frac{1}{r})^{\frac{1}{a_1}}\geq \exp{(n-1)}$$ for every $i_0...i_{n-1}\in \mathcal{A}_r$. It follows that $$e\times(\frac{1}{r})^{\frac{1}{a_1}}\geq \exp{M(r)}.$$ Take $b>0$ such that $e\leq (\frac{1}{r_0})^{b-\frac{1}{a_1}}$. This finishes the proof of (\ref{range}).
		
		To finish the proof, we only need to show that for any $\varepsilon>0$ and $r>0$ small enough, we have
		\begin{equation*}
			r^{\frac{\varepsilon}{2}}\leq \Gamma_r \leq r^{-\frac{\varepsilon}{2}}.
		\end{equation*}
		We first prove that $\Gamma_r \geq r^{\frac{\varepsilon}{2}}$ for $r>0$ small. Suppose this is not true, then we can find some $r>0$ and $\lambda>0$ such that $Z(r,\lambda)<1$, where $$Z(r,\lambda)=\sum_{I\in\mathcal{A}_r}\exp{(\lambda|I|)}\sup_{x\in[I]\cap X}\exp(S_{|I|}\varphi(x)).$$
		By Lemma 2.14 of \cite{B}, it follows that $P_X(\sigma,\varphi)\leq -\lambda$, which is a contradiction.

		By definition we have 
		$$\lim_{n\rightarrow +\infty}\frac{1}{n}\log(\sum_{I\in X^*:|I|=n}\sup_{x\in[I]\cap X}\exp{(S_{|I|}\varphi(x))})=0.$$
		Hence for any $\lambda>0$, there exists a large $N$ such that $e^{-\lambda N/2}<1-e^{-\lambda/2}$ and for any $n\geq N$,
		$$\gamma_n\triangleq \sum_{I\in X^*:|I|=n}\exp(-\lambda|I|)\sup_{x\in [I]\cap X}\exp{(S_{|I|}\varphi(x))}\leq \exp{(-\lambda n/2)}.$$
		Since $m(r)$ tends to infinity as $r$ tends to $0$, where $m(r)$ is defined in (\ref{minimal}), one can take a small $r_1\in (0,r_0)$ so that $m(r)\geq N$ for any $0<r<r_1$. Hence for any $0<r<r_1$, $$\mathcal{A}_r\subset \{I\in X^*:|I|\geq N\}$$ and so 
		\begin{align*}
			&\sum_{I\in\mathcal{A}_r}\exp{(-\lambda|I|)}\sup_{x\in[I]\cap X}\exp(S_{|I|}\varphi(x)) \\
			\leq &\sum_{n=N}^{\infty}\gamma_n \leq \sum_{n=N}^{\infty}\exp{(-\lambda n/2)}=\frac{e^{-\lambda N/2}}{1-e^{-\lambda/2}}<1.
		\end{align*}

		Fix $\lambda\in (0,\frac{\varepsilon}{2b})$. It follows from (\ref{range}) that $r^{\frac{\varepsilon}{2}}\leq \exp{(-\lambda|I|)}$ for any $I\in \mathcal{A}_r$. Therefore
		$$\sum_{I\in\mathcal{A}_r}\sup_{x\in[I]\cap X}\exp(S_{|I|}\varphi(x))<r^{-\frac{\varepsilon}{2}}.$$
	\end{proof}
	\begin{Rem}\label{new2}
		We conjecture that Proposition \ref{unique root} still holds for a general continuous potential, that is, assume that $(X,\sigma)$ is a two-sided subshift over a finite alphabet $\mathcal{A}$,  $\varphi$ and $h$ are continuous potential functions and $h(x)<0$ for all $x\in X$, then we have 
		\begin{eqnarray*}
			P_X(\sigma,\varphi)=\lim_{r\rightarrow 0}\frac{\log \Gamma_r}{-\log r},
		\end{eqnarray*}
		where $$\Gamma_r=\sum_{I\in \mathcal{A}_r}\sup_{x\in [I]\cap X} \exp{(S_{|I|}\varphi(x))}.$$
	\end{Rem}

	\subsection{Proof of the Main Theorem}
	To prove the main theorem, we utilize two elementary results from \cite{FS}. For $T\in \mathbb{R}^{d\times d}$, let $\alpha_1(T)\geq \cdots \geq \alpha_d(T)$ denote the singular values of $T$. For $s\geq 0$, define the singular value function $\phi^s:\mathbb{R}^{d\times d}\rightarrow [0,\infty)$ as 
	\[\phi^s(T)\triangleq\left\{
	\begin{array}{lll}
		\alpha_1(T)\cdots\alpha_k(T)\alpha_{k+1}^{s-k}(T),  &\quad \mbox{if}~ 0\leq s\leq d;\\
		
		(\det{T})^{\frac{s}{d}}, &\quad \mbox{if}~ s>d,
		
	\end{array}\right.\]
	where $k=[s]$ is the integral part of $s$.
	\begin{Lem}\label{cover}
		Suppose $M$ is a $d-$dimensional Riemannian manifold. Let $E\subset U\subset M$, where $E$ is compact and $U$ is open. Let $k\in \{0,1,...,d-1\}$. Then for any non-degenerate $C^1$ map $f: U\rightarrow M$, there exists $r_0>0$ so that for any $y\in E$, $z\in B(y,r_0)$ and $0<r<r_0$, the set $f(B(z,r))$ can be covered by 
		$$C_M\frac{\phi^k(D_yf)}{(\alpha_{k+1}(D_yf))^k}$$
		many balls of radius $\alpha_{k+1}(D_yf)r$, where $C_M$ is a positive constant depending on $M$.
	\end{Lem}
 The proof of the following proposition is similar to that for one-sided shift space in \cite{FS}. For the completeness of the proof and the convenience of the reader, we give the details. Recall that for a locally maximal hyperbolic set $\Lambda$, there exists a two-sided topological Markov chain $\Sigma_A$ and a continuous surjective map $\pi: \Sigma_A\rightarrow \Lambda$ such that $\pi\circ \sigma=f\circ \pi$. Given $\bold{i}=(i_p)_{p=-\infty}^{\infty}\in \Sigma_A$ and $n\in \mathbb{N}$, denote by $P_{\bold{i}|n}=\bigcap^{n-1}_{j=-\infty}f^{-j}P_{i_j}.$
	\begin{Prop}\label{cover2}
		Let $M$ be a compact Riemannian manifold, $f:M\rightarrow M$ be a $C^1$ diffeomorphism and $\Lambda$ be a locally maximal hyperbolic set for $f$. Denote by $u$ the dimension of the unstable manifold. Let $k\in\{0,1,...,u-1\}$. Set $C_M$ the constant in Lemma \ref{cover}. Then there exists $C_1>0$ such that for $\bold{i}=(i_p)_{p=-\infty}^{\infty}\in \Sigma_A$ and $n\in \mathbb{N}$, the cylinder $P_{\bold{i}|n}$ can be covered by $C_1\prod_{p=0}^{n-1}G(\sigma^p\bold{i})$ balls of radius $\prod_{p=0}^{n-1}H(\sigma^p\bold{i})$, where
		$$G(\bold{i})=\frac{C_M\phi^k((D_{\pi\bold{i}}f)^{-1})}{\alpha_{k+1}((D_{\pi\bold{i}}f)^{-1})^k},\quad\quad H(\bold{i})=\alpha_{k+1}((D_{\pi\bold{i}}f)^{-1}).$$
	\end{Prop}
\begin{proof}
	By Lemma \ref{cover}, there exists $r_0>0$ so that for any $y\in \Lambda$, $z\in B(y,r_0)$ and $0<r<r_0$, the set $f^{-1}(B(z,r))$ can be covered by 
	$$C_M\frac{\phi^k(D_yf^{-1})}{(\alpha_{k+1}(D_yf^{-1}))^k}$$
	many balls of radius $\alpha_{k+1}(D_yf^{-1})r$. Since $f$ is expanding on $W^u_{\text{loc}}(x)\cap \Lambda$, there exists $\gamma\in (0,1)$ such that $\sup_{\bold{i}\in \Sigma_A}\alpha_{1}((D_{\pi \bold{i}}f)^{-1})<\gamma$. Then $\sup_{\bold{i}\in \Sigma_A} H(\bold{i})<\gamma$. Since $f^{-1}P_{(\sigma\bold{i})|n}=P_{\bold{i}|(n+1)}$, we have 
	$$\text{diam}(P_{\bold{i}|(n+1)})\leq \gamma \text{diam}(P_{(\sigma \bold{i})|n})$$
	for all $\bold{i}\in \Sigma_A$ and $n\in \mathbb{N}$. Take $n_0\in \mathbb{N}$ large enough such that 
	$$\gamma^{n_0-1}\max\{1,\text{diam}\Lambda\}<\frac{r_0}{2}.$$ Thus $\text{diam}(P_{\bold{i}|n})<\frac{r_0}{2}$ for all $\bold{i}\in \Sigma_A$ and $n\geq n_0$.
	
	Clearly there exists a large number $C_1$ such that the conclusion of the proposition holds for any positive integer $n\leq n_0$ and $\bold {i}\in \Sigma_A$, i.e. the set $P_{\bold{i}|n}$ can be covered by $C_1\prod_{p=0}^{n-1}G(\sigma^p\bold{i})$ balls of radius $\prod_{p=0}^{n-1}H(\sigma^p\bold{i})$. We show this holds for all $n\in \mathbb{N}$ and $\bold{i}\in \Sigma_A$ by induction.
	
	Suppose for some $m\geq n_0$, the conclusion of the proposition holds for any positive integer $n\leq m$ and $\bold{i}\in \Sigma_A$. Then for any given $\bold{i}=(i_p)_{p=-\infty}^{\infty}\in \Sigma_A$, $P_{(\sigma\bold{i})|m}$ can be covered by $C_1\prod_{p=0}^{m-1}G(\sigma^{p+1}\bold{i})$ balls of radius $\prod_{p=0}^{m-1}H(\sigma^{p+1}\bold{i})$. Let $B_1,...B_N$ denote these balls. We assume that $B_j\cap P_{(\sigma\bold{i})|m}\not=\varnothing$ for each $j$. Since 
	\begin{align}
		\prod_{p=0}^{m-1}H(\sigma^{p+1}\bold{i})\leq \gamma^m\leq \gamma^{n_0}<\frac{r_0}{2}
	\end{align}
and 
\begin{align}
	d(\pi(\sigma\bold{i}),B_j\cap P_{(\sigma\bold{i})|m})\leq \text{diam}(P_{(\sigma\bold{i})|m})<\frac{r_0}{2},
\end{align}
so the center of $B_j$ is in $B(\pi (\sigma\bold{i}),r_0)$. Therefore, by Lemma \ref{cover}, $f^{-1}(B_j)$ can be covered by 
\begin{align}
	\frac{C_M\phi^k(D_{\pi (\sigma\bold{i})}f^{-1})}{(\alpha_{k+1}(D_{\pi (\sigma\bold{i})}f^{-1}))^k}=\frac{C_M\phi^k((D_{\pi\bold{i}}f)^{-1})}{\alpha_{k+1}((D_{\pi\bold{i}}f)^{-1})^k}=G(\bold{i})
\end{align}
balls of radius 
\begin{align}
	\alpha_{k+1}(D_{\pi(\sigma\bold{i})}f^{-1})\times \prod_{p=0}^{m-1}H(\sigma^{p+1}\bold{i})=H(\bold{i})\times\prod_{p=0}^{m-1}H(\sigma^{p+1}\bold{i})=\prod_{p=0}^{m}H(\sigma^p\bold{i}).
\end{align}
Since ${P_{\bold{i}|(m+1)}}=f^{-1} P_{(\sigma \bold{i})|m}$, it follows that 
\begin{align}
	P_{\bold{i}|(m+1)}= f^{-1}P_{(\sigma\bold{i})|m}\subset\bigcup_{j=1}^{N} f^{-1}B_j,
\end{align}
hence $P_{\bold{i}|(m+1)}$ can be covered by
\begin{align}
	G(\bold{i})N\leq G(\bold{i})C_1\prod_{p=0}^{m-1}G(\sigma^{p+1}\bold{i})=C_1\prod_{p=0}^{m}G(\sigma^p\bold{i})
\end{align}
balls of radius $\prod_{p=0}^m H(\sigma^p\bold{i})$. Thus the proposition also holds for $n=m+1$ and all $\bold{i}\in \Sigma_A$.
\end{proof}

	The following result gives an explicit expression for the unique root of the Bowen equation.
	\begin{Prop}\label{root2}
		Let $(X,\sigma)$ be a two-sided subshift over a finite alphabet $\mathcal{A}$ and $g,h$ be continuous functions. Assume  $h(x)<0$ for all $x\in X$. Let $r_0=\sup_{x\in X}\exp{(h(x))}$. Set for $0<r<r_0$,
		$$\mathcal{A}_r\triangleq \{i_0...i_{n-1}\in X^\ast:\sup_{x\in[i_0...i_{n-1}]\cap X}\exp{(S_nh(x))}<r\leq \sup_{y\in[i_0...i_{n-2}]\cap X}\exp{(S_{n-1}h(y))}\},$$
		where $X^*$ is the collection of finite words allowed in $X$. Then
		\begin{align}
			\lim_{r\rightarrow 0}\frac{\log \sum_{I\in\mathcal{A}_r}\sup_{x\in[I]\cap X}\exp{(S_{|I|}g(x))}}{\log r}=-t,
		\end{align}
		where $t$ is the unique root of $P_X(\sigma,g+th)=0$.
	\end{Prop}
	To prove the proposition, we need the following lemma, which is well-known. For a proof of the one-sided subshift, one can refer to \cite{FS}.
	\begin{Lem}\label{Bowen P}
		Let $(X,\sigma)$ be a two-sided subshift over a finite alphabet $\mathcal{A}$ and $f$ be a continuous function. Then 
		\begin{align*}
			\lim_{n\rightarrow \infty}\frac{1}{n}\sup\{|S_nf(x)-S_nf(y)|:x_i=y_i ~\text{for all} ~i\leq n\}=0.
		\end{align*} 
	\end{Lem}
	\begin{proof}[Proof of Proposition \ref{root2}]
		For $\varepsilon >0$ and $r>0$ small enough, by Proposition \ref{unique root}, we have $$r^{\frac{\varepsilon}{2}}\leq \Gamma_r\leq r^{-\frac{\varepsilon}{2}},$$ where
		$\Gamma_r=\sum_{I\in\mathcal{A}_r}\sup_{x\in[I]\cap X}\exp{(S_{|I|}(g+th)(x))}$. By Lemma \ref{Bowen P} and (\ref{range}), we have $$r^{t+\frac{\varepsilon}{2}}\Theta_r\leq \Gamma_r \leq r^{t-\frac{\varepsilon}{2}}\Theta_r,$$ where $\Theta_r=\sum_{I\in\mathcal{A}_r}\sup_{x\in[I]\cap X}\exp{(S_{|I|}g(x))}$. Thus we have $$r^{-t+\varepsilon}\leq \Theta_r \leq r^{-t-\varepsilon}.$$
		This finishes the proof.
	\end{proof}

	We use Proposition \ref{cover2} to estimate the upper box dimension of $W^u_{\text{loc}}(x)\cap \Lambda$.
	\begin{Prop}\label{bound 1}
		Let $k\in \{0,1,...,u-1\}$. Let $G,H:\Sigma_A\rightarrow \mathbb{R}$ be defined as in Proposition \text{\ref{cover2}}. Let $t$ be the unique root of 
		\begin{equation*}
			P(\sigma|_{\Sigma_A},\log G+t\log H)=0.
		\end{equation*}
		Then for any $x\in \Lambda$, we have $\overline{dim}_B  W^u_{\text{loc}}(x)\cap \Lambda \leq t$.
	\end{Prop}
	\begin{proof}
		Let $g=\log G$ and $h=\log H$. Define 
		\begin{equation*}
			r_{\text{min}}=\min_{\bold{i}\in\Sigma_A} h(\bold{i}).
		\end{equation*}
		Then $r_{\text{min}}>0$. For $0<r<r_{\text{min}}$, define 
		$$\mathcal{A}_r\triangleq \{i_0...i_{n-1}\in \Sigma_A^\ast:\sup_{x\in[i_0...i_{n-1}]\cap \Sigma_A}\exp{(S_nh(x))}<r\leq \sup_{y\in[i_0...i_{n-2}]\cap \Sigma_A}\exp{(S_{n-1}h(y))}\},$$
		where $\Sigma_A^*$ denotes the set of all finite words allowed in $\Sigma_A$. By Proposition \ref{cover2}, there exists $C_1>0$ such that for any $0<r<r_{\text{min}}$, every $I\in\mathcal{A}_r$ and $x\in[I]\cap\Sigma_A$, $P_I$ can be covered by
		$$C_1 \exp{(S_{|I|}g(x))}\leq C_1 \exp{(\sup_{y\in[I]\cap\Sigma_A}S_{|I|}g(y))}$$
		many balls of radius
		$$\exp{(S_{|I|}h(x))}\leq \exp{(\sup_{y\in[I]\cap\Sigma_A}S_{|I|}g(y)})<r.$$
		It follows that $W^u_{\text{loc}}(x)\cap \Lambda$ can be covered by 
		$$C_1\sum_{I\in\mathcal{A}_r}\exp{(\sup_{y\in[I]\cap\Sigma_A}S_{|I|}g(y))}$$ 
		many balls of radius $r$. Hence by Proposition \ref{root2},
		\begin{equation*}
			\overline{\text{dim}}_B  W^u_{\text{loc}}(x)\cap \Lambda\leq \limsup_{r\rightarrow 0} \frac{\sum_{I\in\mathcal{A}_r}\exp(\sup_{y\in[I]\cap\Sigma_A}g(y))}{\log(\frac{1}{r})}=t.
		\end{equation*}
	\end{proof}
	The following lemma allows us to compare the entropy on the symbolic space and that on the hyperbolic set. See \cite{FS} for the proof.
	\begin{Lem}\label{entropy}
		Let $X_i,i=1,2$ be compact metric spaces and let $T_i:X_i\rightarrow X_i$ be continuous. Suppose $\pi:X_1\rightarrow X_2$ is a continuous surjection such that the following diagram commutes: $\pi\circ T_1=T_2\circ\pi$. Then $\pi_{\ast}:\mathcal{M}_{inv}(T_1)\rightarrow \mathcal{M}_{inv}(T_2)$ (defined by $\mu\rightarrow \mu\circ\pi^{-1}$) is surjective. If furthermore there is an integer $q>0$ so that $\pi^{-1}(y)$ has at most $q$ elements for each $y\in X_2$, then 
		$$h_{\mu}(T_1)=h_{\mu\circ\pi^{-1}}(T_2)$$
		for each $\mu\in \mathcal{M}_{inv}(T_1)$.
	\end{Lem}
	Next we give the proof of the Main Theorem. This is inspired by the technique in \cite{BCH,FS}. 
	\begin{proof}[Proof of Theorem \ref{Main A}]
		Denote by $s^*$ the unique root of $P^u(f|_\Lambda,\Phi_f(s))=0$. By Theorem \ref{equal}, $s^*$ is also the root of $P(f|_{\Lambda},\Phi_f(s))=0$. Assume that $s^*<u$, where $u$ is the dimension of the unstable manifold, otherwise there is nothing to prove. Set $k=[s^*]$, namely the largest integer less than or equal to $s^*$. Since the factor map $\pi: \Sigma_{A} \rightarrow \Lambda$ is onto and finite-to-one, $m\rightarrow m\circ \pi^{-1}$ is a surjective map from $\mathcal{M}_{inv}(\Sigma_A,\sigma^n)$ to $\mathcal{M}_{inv}(\Lambda,f^n)$ for each $n\in\mathbb{N}$ and moreover, by Lemma \ref{entropy}, one has $h_m(\sigma^n)=h_{m\circ \pi^{-1}}(f^n)$ for $m\in \mathcal{M}_{inv}(\Sigma_A,\sigma^n)$. 

		For each $n\in\mathbb{N}$, $\Lambda$ is also a hyperbolic set for $f^{2^n}$. By the variational principle, for each $s\in [0,u]$, one has 
		\begin{align*}
			&P(f^{2^n}|_\Lambda,-\varphi^s(\cdot,f^{2^n}))\\
			=&\sup\{h_{\mu}(f^{2^n})+\int-\varphi^s(\cdot,f^{2^n})d\mu :\mu\in\mathcal{M}_{inv}(\Lambda,f^{2^n})\}\\
			=&\sup\{h_{m\circ\pi^{-1}}(f^{2^n})+\int-\varphi^s(\cdot,f^{2^n})d m\circ\pi^{-1}:m\in\mathcal{M}_{inv}(\Sigma_A,\sigma^{2^n})\}\\
			=&\sup\{h_m(\sigma^{2^n})+\int \log G_{2^n} +s\log H_{2^n}d m:m\in\mathcal{M}_{inv}(\Sigma_A,\sigma^{2^n})\}\\
			=&P(\sigma^{2^n}|_{\Sigma_A},\log G_{2^n} +s\log H_{2^n}),	
		\end{align*}
		where $G_{2^n}(\bold{i})=\frac{C\phi^k((D_{\pi \bold{i} }f^{2^n})^{-1})}{\alpha_{k+1}((D_{\pi \bold{i}}f^{2^n})^{-1})^k}$ and $H_{2^n}(\bold{i})=\alpha_{k+1}((D_{\pi \bold{i}}f^{2^n})^{-1})$. 
		By Proposition \ref{bound 1}, we have 
		$$	\overline{\text{dim}}_B  W^u_{\text{loc}}(x,f^{2^n}) \cap \Lambda\leq t_n,$$
		where $t_n$ is the unique root of 
		$$P(f^{2^n}|_\Lambda,-\varphi^s(\cdot,f^{2^n}))=0.$$
		One can show that $t_n\geq t_{n+1}$.  In detail, for $\ell\in\mathbb{N}$, set $\phi^s_\ell=-\varphi^s(\cdot,f^\ell)$.  Recall that 
		$$P_n(\phi^s_\ell,\varepsilon)=\sup\{\sum_{x\in E}\exp{S_n\phi^s_\ell(x)}: E \text{~is an }(n,\varepsilon)\text{-separated subset of~} M\}.$$
		For any $\varepsilon>0$, by the uniform continuity of $f$, there exists $\delta>0$ such that if $E\subset M$ is an ($n,\varepsilon$)-separated set of $f^{2^{\ell+1}}$, then $E$ is a ($2n,\delta$)-separated set of $f^{2^\ell}$ and $\delta\rightarrow 0$ when $\varepsilon \rightarrow 0$. By the sub-additivity of $\phi_\ell^s$, the Birkhoff sum $S_n\phi^s_{2^{\ell+1}}$ of $\phi^s_{2^{\ell+1}}$ with respect to $f^{2^{\ell+1}}$ has the following property:
		\begin{align*}
			S_n\phi^s_{2^{\ell+1}}(x)&=\phi^s_{2^{\ell+1}}(x)+\phi^s_{2^{\ell+1}}(f^{2^{\ell+1}}(x))+\cdots+\phi^s_{2^{\ell+1}}(f^{2^{\ell+1}(n-1)}(x))\\
			&\leq \phi^s_{2^{\ell}}(x)+\phi^s_{2^{\ell}}(f^{2^{\ell}}(x))+\phi^s_{2^{\ell}}(f^{2^{\ell+1}}(x))+\phi^s_{2^{\ell}}(f^{2^{\ell+1}}(f^{2^\ell}(x)))+\\
			&\cdots+\phi^s_{2^{\ell}}(f^{2^{\ell+1}(n-1)}(x))+\phi^s_{2^{\ell}}(f^{2^{\ell+1}(n-1)}(f^{2^\ell}(x)))\\
			&=S_{2n}\phi^s_{2^\ell}(x),
		\end{align*}
		where $S_{2n}\phi^s_{2^\ell}(x)$ is the Birkhoff sum of $\phi^s_{2^\ell}$ with respect to $f^{2^\ell}$. Thus 
		$$P_n(f^{2^{\ell+1}}|_{\Lambda},\phi^s_{2^{\ell+1}},\varepsilon)\leq P_{2n}(f^{2^{\ell}}|_{\Lambda},\phi^s_{2^{\ell}},\delta).$$
		Hence, $$P(f^{2^{\ell+1}}|_{\Lambda},\phi^s_{2^{\ell+1}})\leq 2P(f^{2^\ell}|_{\Lambda},\phi^s_{2^\ell}).$$
		Take $s=t_{\ell+1}$, then we have
		$$0=P(f^{2^{\ell+1}}|_{\Lambda},\phi^{t_{\ell+1}}_{2^{\ell+1}})\leq 2P(f^{2^\ell}|_{\Lambda},\phi^{t_{\ell+1}}_{2^\ell}).$$
		It follows that  $t_\ell\geq t_{\ell+1}$. Thus the limitation $t^*=\lim_{n\rightarrow \infty} t_n$ exists.

		Next we show that $t^*$ is the unique root of $P(f|_\Lambda,\Phi_f(s))=0$.
		By the continuity of $P(f|_\Lambda,\Phi_f(s))$ with respect to $s$, it holds that 
		$$P(f|_\Lambda,\Phi_f(t^*))=\lim_{n\rightarrow \infty} P(f|_\Lambda,\Phi_f(t_n)).$$
		By Proposition \ref{topo pressure appro}, for each $n\in \mathbb{N}$, it holds that 
		$$P(f|_\Lambda,\Phi_f(t_n))=\lim_{\ell\rightarrow\infty}\frac{1}{2^\ell}P(f^{2^\ell}|_\Lambda,-\varphi^{t_n}(\cdot,f^{2^\ell})).$$
		By the monotonicity of $t_n$, for each $n\in\mathbb{N}$, it holds that 
		$$\lim_{\ell\rightarrow\infty}\frac{1}{2^\ell}P(f^{2^\ell}|_\Lambda,-\varphi^{t_n}(\cdot,f^{2^\ell}))\leq 0.$$
		It follows that $P(f|_\Lambda,\Phi_f(t_n))\leq 0$ and hence $P(f|_\Lambda,\Phi_f(t^*))\leq 0$. On the other hand, 
		$$P(f|_\Lambda,\Phi_f(t^*))=\lim_{\ell\rightarrow\infty}\frac{1}{2^\ell}P(f^{2^\ell}|_\Lambda,-\varphi^{t^*}(\cdot,f^{2^\ell}))\geq 0.$$
		Thus $t^*$ is the root of the equation $P(f|_\Lambda,\Phi_f(s))=0$. By Theorem \ref{equal}, $t^*$ is also the root of $P^u(f|_\Lambda,\Phi_f(s))=0$. This completes the proof.
	\end{proof}
	\section{Appendix}\label{appendix}
	In this section, we give an upper bound of the Hausdorff dimension of the hyperbolic set on unstable manifold, which partially inspires the main theorem in this paper.
	
	Given a $C^1$ diffeomorphism $f$ and a hyperbolic set $\Lambda$ for $f$. We still denote $\varphi^{s}(\cdot, f^{n})$ for $s\in[0,u]$ as the singular-valued function as that in Section \ref{Preliminary}. Consider the following topological pressure function
	$$P_{n}^{u}(s)\triangleq P(f|_{\Lambda,}-\frac{1}{n}\varphi^{s}(\cdot, f^{n})), \quad s\in[0,u].$$
	
	By the sub-additivity of $\{-\varphi^{s}(\cdot, f^{n})\}_{n\in\mathbb{N}}$, Zhang \cite{Zh} proved the following result:
	\begin{Lem}
		For every $s\in [0,u]$, the following limitation exists
		$$\lim_{n\rightarrow \infty} P_{n}^{u}(s)=\inf_{n\in\mathbb{N}} P_{n}^{u}(s).$$
		If we denote $$P^{\ast}(s)=\lim_{n\rightarrow \infty} P_{n}^{u}(s),$$ then $P^{\ast}(s)$ is continuous and decreasing with $s$. $P^{\ast}(0)=h_{\text{top}}(f|_{\Lambda})\geq 0$.
	\end{Lem}

	Zhang \cite{Zh} proved  that
	\begin{Lem}\label{upperboundlemma}
		Denote $$D_{1}^{u}=\max\{s\in[0,u]:P_{1}^{u}(s)\geq 0\}.$$ Then for every $x\in\Lambda$, we have 
		$$\text{dim}_{H}W^{u}_{\text{loc}}(x)\cap\Lambda\leq D^{u}_{1}.$$
	\end{Lem}
	
	The main theorem in \cite{Zh} is the following.
	\begin{Thm}\label{upperbound*}
		Denote $$D^{u}(f,\Lambda)=\max\{s\in[0,u]:P^{\ast}(s)\geq 0\}.$$
		Then for every ~$x\in \Lambda$, we have 
		$$\text{dim}_{H}W^{u}_{\text{loc}}(x)\cap \Lambda\leq D^{u}(f,\Lambda).$$
	\end{Thm}
	\begin{Rem}\label{root}
		In fact, $D_{1}^{u}$ is exactly the unique root of $P_{1}^{u}(s)=0$. By Ruelle's inequality and variational principle, we have $P_{1}^{u}(u)\leq 0$. Besides, $P_{1}^{u}(0)>0$. Similarly, $D^{u}(f,\Lambda)$ is the unique root of $P^{\ast}(s)=0$, for this, one can refer to Remark 1 in \cite{BCH} for repellers.
	\end{Rem}

	From now on, we assume that $\Lambda$ is topologically transitive. In Zhang's above result, $D_{1}^{u}$ is defined via the topological pressure on the whole hyperbolic set. We want to study the dimension of the hyperbolic set on the local unstable manifold via the pressure on the unstable leaf. We define the quantity $$\tilde{D}_{1}^{u}=\max\{s\in[0,u]:P_{W^{u}_{\text{loc}}(x)\cap\Lambda}(f,-\varphi^{s}(\cdot,f))\geq 0\}.$$ By the monotonicity of the topological pressure on subsets, we have $\tilde{D}_{1}^{u}\leq D_{1}^{u}$. A natural question is whether $\tilde{D}_{1}^{u}$ is an upper bound for the Hausdorff dimension of the hyperbolic set on the local unstable manifold. We prove that in fact we have $\tilde{D}_{1}^{u}= D_{1}^{u}$, which answers this question positively.

	First, we prove the following result:
	\begin{Thm}\label{pressure on unstable leaf}
		Suppose $g:M\rightarrow M$ is a $C^{1}$ diffeomorphism and $\Lambda_{g}$ is a hyperbolic set for $g$. Then for every $x\in\Lambda_{g}$, $n\in\mathbb{N}$ and $s\in[0,u]$, we have 
		$$P_{\Lambda_{g}}(g,-\frac{1}{n}\varphi^{s}(\cdot,g^{n}))=P_{W^{u}_{\text{loc}}(x)\cap\Lambda_{g}}(g,-\frac{1}{n}\varphi^{s}(\cdot,g^{n})).$$
	\end{Thm}
	
	With this result and Lemma \ref{upperboundlemma} and Remark \ref{root}, we have the following result:
	\begin{Thm}\label{upper bound for hyperbolic_main1}
		For $x\in\Lambda$, denote by $\tilde{D}_{1}^{u}$ the unique root of $P_{W^{u}_{\text{loc}}(x)\cap\Lambda}(f|_{\Lambda},-\varphi^{s}(\cdot,f))=0$. We have
		$$\text{dim}_{H}W^{u}_{\text{loc}}(x)\cap\Lambda\leq \tilde{D}^{u}_{1}.$$
	\end{Thm}
	
	Denote $\tilde{P}^{\ast}(s)=\lim_{n\rightarrow \infty}P_{W^{u}_{\text{loc}}(x)\cap\Lambda}(f|_{\Lambda,}-\frac{1}{n}\varphi^{s}(\cdot, f^{n})).$
	The main result in this section is the following theorem:
	\begin{Thm}\label{upper bound for hyperbolic_main2}
		For $x\in \Lambda$, denote by $\tilde{D}^{u}(f,\Lambda)$ the unique root of $\tilde{P}^{\ast}(s)=0$. We have
		$$\text{dim}_{H}W^{u}_{\text{loc}}(x)\cap \Lambda\leq \tilde{D}^{u}(f,\Lambda).$$
	\end{Thm}

	\begin{Rem}\label{invalid}
		By Remark $\ref{root}$ and Proposition $\ref{topo pressure appro}$, the upper bound given by Theorem $\ref{upperbound*}$ is exactly the root of sub-additive topological pressure on the hyperbolic set for the sub-additive singular-valued potential. The proof of Proposition $\ref{topo pressure appro}$ requires the variational principle of the topological pressure for the sub-additive potential, which does not hold for the topological pressure on a general non-compact set. Thus it seems that it is difficult to  use the root of sub-additive topological pressure on the local unstable manifold to give an upper bound of the Hausdorff dimension of the hyperbolic set on the unstable direction. 
	\end{Rem}
	By Theorem \ref{pressure on unstable leaf}, we have $D^{u}_{1}=\tilde{D}^{u}_{1}$ and  $D^{u}(f,\Lambda)=\tilde{D}^{u}(f,\Lambda)$. Therefore, by Lemma \ref{upperboundlemma} and Theorem \ref{upperbound*}, we finish the proof of Theorem \ref{upper bound for hyperbolic_main1} and \ref{upper bound for hyperbolic_main2}. So we only need to prove Theorem \ref{pressure on unstable leaf}. For this, we use a result of Climenhaga, Pesin and Zelerowicz \cite{CliPZ}.
	\begin{Thm}\label{pressure}
		Given a $C^{1+\alpha}$ diffeomorphism $f$ and a topologically transitive hyperbolic set $\Lambda$ for $f$. For a H$\ddot{\text{o}}$lder continuous potential $\varphi:\Lambda\rightarrow \mathbb{R}$ and any $x\in \Lambda$, we have 
		$$P_{\Lambda}(f,\varphi)=P_{W^{u}_{\text{loc}}(x)\cap\Lambda}(f,\varphi).$$
	\end{Thm}
	
	We utilize the structural stability of the hyperbolic set and Theorem $\ref{pressure}$ to prove Theorem \ref{pressure on unstable leaf}.
	
	\begin{proof}{[Proof of Theorem \ref{pressure on unstable leaf}]}

		For $y\in M$ and $n\in\mathbb{N}$, denote the singular values of $D_yg^{n}$ decreasingly by
		$$\alpha_{1}(y,g^{n})\geq\cdots\geq\alpha_{d}(y,g^{n}),$$
		where $d$ is the dimension of the manifold $M$.
		For every $s\in [0,u]$, denote
		$$\varphi^{s}(y, g^{n})\triangleq\sum_{i=u-[s]+1}^{u}\log \alpha_{i}(y,g^{n})+(s-[s])\log\alpha_{u-[s]}(y,g^{n}).$$
		For $\varepsilon_{1}>0$, there exists $\delta>0$, such that for any $x_{1},x_2\in M$ satisfying  $d(x_{1},x_{2})\leq\delta$, we have 	$$|\frac{1}{n}\varphi^{s}(x_{1},g^{n})-\frac{1}{n}\varphi^{s}(x_{2},g^{n})|\leq \frac{\varepsilon_{1}}{4}.$$
		By Theorem \ref{structure stability}, for $\delta>0$ above, there exists $\varepsilon>0$ such that for any diffeomorphism $f:M\rightarrow M$ with $d_{C^{1}}(f,g)<\varepsilon$ and $d_{C^{1}}(f^n,g^n)<\varepsilon$, there exists  a hyperbolic set $\Lambda_{f}$ for $f$ and a homeomorphism $h:\Lambda_{f}\rightarrow \Lambda_{g}$, such that $d_{C^{0}}(Id,h)+d_{C^{0}}(Id,h^{-1})<\delta$ and for $x\in \Lambda_f$, one has $$|\frac{1}{n}\varphi^{s}(x,g^{n})-\frac{1}{n}\varphi^{s}(x,f^{n})|\leq \frac{\varepsilon_{1}}{4}.$$
		Thus
		\begin{align*}
			&\max_{x\in\Lambda_{f}}|\frac{1}{n}\varphi^{s}(h(x),g^{n})-\frac{1}{n}\varphi^{s}(x,f^{n})|\\
			\leq & \max_{x\in\Lambda_{f}}(|\frac{1}{n}\varphi^{s}(h(x),g^{n})-\frac{1}{n}\varphi^{s}(x,g^{n})|+|\frac{1}{n}\varphi^{s}(x,g^{n})-\frac{1}{n}\varphi^{s}(x,f^{n})|)\\
			\leq &\frac{\varepsilon_{1}}{2}.
		\end{align*}
		Choose a $C^{1+\alpha}$ diffeomorphism $f$ satisfying~$d_{C^{1}}(f,g)<\varepsilon$ and $d_{C^{1}}(f^n,g^n)<\varepsilon$. For $x\in\Lambda_{g}$, by Theorem \ref{pressure}, we have
		\begin{equation}\label{pressure relation4}
			P_{\Lambda_{f}}(f,-\frac{1}{n}\varphi^{s}(\cdot,f^{n}))=P_{W^{u}_{\text{loc}}(h^{-1}(x))\cap\Lambda_{f}}(f,-\frac{1}{n}\varphi^{s}(\cdot,f^{n})).
		\end{equation}	
		By the property of topological pressure, we have 
		\begin{equation}\label{pressure relation5}
			P_{\Lambda_{g}}(g,-\frac{1}{n}\varphi^{s}(\cdot,g^{n}))=P_{\Lambda_{f}}(f,-\frac{1}{n}\varphi^{s}(h(\cdot),g^{n}))
		\end{equation}
		and 
		\begin{equation}\label{pressure relation6}
			P_{W^{u}_{\text{loc}}(x)\cap \Lambda_{g}}(g,-\frac{1}{n}\varphi^{s}(\cdot,g^{n}))
			=P_{W^{u}_{\text{loc}}(h^{-1}(x))\cap\Lambda_{f}}(f,-\frac{1}{n}\varphi^{s}(h(\cdot),g^{n})).
		\end{equation}
		By (\ref{pressure relation4}), (\ref{pressure relation5}) and (\ref{pressure relation6}), we have
		\begin{align*}
			&|P_{\Lambda_{g}}(g,-\frac{1}{n}\varphi^{s}(\cdot,g^{n}))-P_{W^{u}_{\text{loc}}(x)\cap\Lambda_{g}}(g,-\frac{1}{n}\varphi^{s}(\cdot,g^{n}))|\\
			=&|P_{\Lambda_{f}}(f,-\frac{1}{n}\varphi^{s}(h(\cdot),g^{n}))-P_{W^{u}_{\text{loc}}(x)\cap\Lambda_{g}}(g,-\frac{1}{n}\varphi^{s}(\cdot,g^{n}))|\\
			\leq &|P_{\Lambda_{f}}(f,-\frac{1}{n}\varphi^{s}(h(\cdot),g^{n}))-P_{\Lambda_{f}}(f,-\frac{1}{n}\varphi^{s}(\cdot,f^{n}))|\\
			+&|P_{\Lambda_{f}}(f,-\frac{1}{n}\varphi^{s}(\cdot,f^{n}))-P_{W^{u}_{\text{loc}}(h^{-1}(x))\cap\Lambda_{f}}(f,-\frac{1}{n}\varphi^{s}(\cdot,f^{n}))|\\
			+&|P_{W^{u}_{\text{loc}}(h^{-1}(x))\cap\Lambda_{f}}(f,-\frac{1}{n}\varphi^{s}(\cdot,f^{n}))-P_{W^{u}_{\text{loc}}(h^{-1}(x))\cap\Lambda_{f}}(f,-\frac{1}{n}\varphi^{s}(h(\cdot),g^{n}))|\\
			+&|P_{W^{u}_{\text{loc}}(h^{-1}(x))\cap\Lambda_{f}}(f,-\frac{1}{n}\varphi^{s}(h(\cdot),g^{n}))-P_{W^{u}_{\text{loc}}(x)\cap\Lambda_{g}}(g,-\frac{1}{n}\varphi^{s}(\cdot,g^{n}))|\\
			\leq &2\max_{x\in\Lambda_{f}}|\frac{1}{n}\varphi^{s}(h(x),g^{n})-\frac{1}{n}\varphi^{s}(x,f^{n})|\leq \varepsilon_{1}.
		\end{align*}
		Letting $\varepsilon_{1}$ tend to $0$, we finish the proof.
	\end{proof}
	\begin{Rem}
		Since the potential function in Theorem $\ref{pressure}$ requires H$\ddot{\text{o}}$lder continuity, $(\ref{pressure relation4})$ does not hold for a general continuous potential $\varphi:M\rightarrow \mathbb{R}$. Thus our proof is invalid for a general continuous potential $\varphi:M\rightarrow \mathbb{R}$. 
	\end{Rem}

\end{document}